\documentclass{elsarticle}

\usepackage{enumerate} 
\usepackage{graphicx}
\usepackage{subfigure}
\usepackage{latexsym}
\usepackage{amsmath}
\usepackage{amssymb}
\usepackage{amsfonts}
\usepackage{verbatim}
\usepackage{mathrsfs}
\usepackage[latin1]{inputenc}
\usepackage{color}
\usepackage{caption}

\newtheorem{tm}{Theorem}[section] 
\newtheorem{rk}{Remark}[section]

\newtheorem{prop}{Proposition}[section]
\newtheorem{lm}{Lemma}[section]

\newproof{proof}{Proof}

\numberwithin{figure}{section} 

\newcommand{\bi} {\mathbf i}
\newcommand{\E}{\mathbb E}
\newcommand{\PP}{\mathbb P}
\newcommand{\N}{\mathbb N}

\newcommand{\R}{\mathbb R}

\newcommand{\CC}{\mathcal C}

\newcommand{\FFF}{\mathscr F}

\begin{document}

\begin{frontmatter}

\title{Stochastic Symplectic and Multi-Symplectic Methods for Nonlinear Schr\"odinger Equation with White Noise Dispersion}

\author[cas]{Jianbo Cui\corref{cor}}
\ead{jianbocui@lsec.cc.ac.cn}

\author[cas]{Zhihui Liu}
\ead{liuzhihui@lsec.cc.ac.cn}

\author[cas]{Jialin Hong}
\ead{hjl@lsec.cc.ac.cn}

\author[nudt]{Weien Zhou}
\ead{weienzhou@nudt.edu.cn}

\cortext[cor]{Corresponding author.}

\address[cas]{Institute of Computational Mathematics and Scientific/Engineering Computing, Chinese Academy of Sciences, Beijing, 100190, China}

\address[nudt]{College of Science, National University of Defense Technology, Changsha 410073, China}


\begin{abstract}
We indicate that the nonlinear Schr\"odinger equation with white noise dispersion possesses stochastic symplectic and multi-symplectic structures. Based on these structures, we propose the stochastic symplectic and multi-symplectic methods, which 
preserve the continuous and discrete charge conservation laws, respectively.
Moreover, we show that the proposed methods are convergent with temporal order one in probability. 
Numerical experiments are presented to verify our theoretical results.
\end{abstract}

\end{frontmatter}

\section{Introduction}

\iffalse
Due to the importance of optical fibers in modern communication systems, They have been extensively studied (see \cite{Mar06}). As a frequently occurring phenomenon, the chromatic dispersion attracts many engineers, physicists, and applied mathematicians.
\fi

In dispersion-managed fibers, the random nonlinear Schr\"odinger equation
\begin{align}\label{l-nls}
&\bi\, du+\epsilon^{-1}m (t\epsilon^{-2})\Delta u\,dt+|u|^{2\sigma} u\,dt=0,\quad \epsilon>0,\quad \text{in} \quad\R_+\times \R
\end{align}
describes the electric field evolution (see \cite{Mar06}).
Here $\bi$ denotes the imaginary unit, $m$ is a real-valued centered stationary continuous random process, and the nonlinear term models the nonlinear response of the medium to the electric field. 
Under certain assumptions on $m$ and $\sigma$, it is proved in \cite{DD10} that the limit equation of \eqref{l-nls}, when $\epsilon$ goes to 0, is the nonlinear Schr\"odinger (NLS) equation with white noise dispersion: 
\begin{align}\label{s-nls}
&\bi\, du+\Delta u\circ dW(t)+|u|^{2\sigma} u \,dt=0,\quad \text{in} \quad\R_+\times \R 
,
\end{align}
where $W$ is a one-dimensional {\color{blue}real-valued} Brownian motion on a stochastic basis $(\Omega,\FFF,(\FFF_t)_{t\ge0},\PP)$.
Recently, many attentions have been paid to the studies of Eq. \eqref{s-nls} both theoretically and numerically (see \cite{BDD15, CD16, DD10, DT11}).
The local and global existence of the unique solution for Eq. \eqref{s-nls} is  established in \cite{DD10, DT11}.
To investigate the inherent nature of Eq. \eqref{s-nls}, authors in \cite{BDD15} numerically study Eq. \eqref{s-nls}  by a Crank-Nicolson scheme,
and obtain first order of convergence in probability for this scheme provided $u_0\in H^{s+10}$ with $s>\frac 12$.
%\cite{CD16} uses exponential integrators to discrete Eq. \eqref{s-nls}.

%Since the qualitative features play crucial roles in both theroetical and nurmerical analysis for NLS equations (see e.g. \cite{BDD15,CH16,DD03,DD04,DD06,JWH13}), it is natural to wonder whether Eq. \eqref{s-nls} posseses some similar features.
%However, due to  the fact that the Hamiltonian structure of the deterministic NLS equation is destoried by the white noise dispersion, 
%%the qualitative  behavior of solutions of Eq. \eqref{s-nls} has not been extensively studied (see \cite{BDD15}). 
%to the best of our knowledge, there has been no work cencerning qualitative features  for Eq. \eqref{s-nls} except the charge conservation law obtained in \cite{DD10} .
%In this paper, we indicate that Eq. \eqref{s-nls}  preserves both the stochastic symplectic and multi-symplectic structures, which are significant qualitative features characterizing  the invariants of the phase flow (see e.g. \cite{CH16,JWH13}).
%As Eq. \eqref{s-nls}  can not be solved analytically,  it is necessary to design numerical methods for this equation.  
%In order to guarantee the reliability and effectiveness of numerical solutions, especially for longtime simulations, we should develop numerical methods  preserving the intrinsic properties of the original problems as much as possible. 

As the qualitative features, such as conserved quantities, symplecticity and multi-symplecticity, play crucial roles in both theoretical and numerical analysis for deterministic NLS equations (see e.g \cite{HLMZ06,IKS01,Kato}), it is natural to wonder whether stochastic NLS equations possess these features.  For the stochastic NLS equation with a stochastic forcing 
\begin{align}\label{nlsf}
\bi\, du+\Delta u \,dt+\lambda |u|^{2\sigma}u\,dt+u\circ dW=0,
\end{align}
where $\lambda\in \R$ and $W$ is a infinite-dimensional Wiener process, 
\cite{DD03} presents the charge conservation law and the evolution of the energy which are used to study the well-posedness of the solution.
\cite{CH16} and \cite{JWH13}  propose the stochastic symplectic  and  multi-symplectic structures  for Eq. \eqref{nlsf}, respectively, which characterize the geometric invariants of the phase flow. 
%the qualitative  behavior of solutions of Eq. \eqref{s-nls} has not been extensively studied (see \cite{BDD15}). 
To the best of our knowledge, there has been little work concerning the geometric structures of Eq. \eqref{s-nls}.
%One objective of this paper is to investigate the stochastic symplectic and multi-symplectic structures of Eq. \eqref{s-nls}.
In this paper, we show that Eq. \eqref{s-nls}  possesses both stochastic symplectic and multi-symplectic structures. In particular, the 
stochastic multi-symplectic conservation law for Eq. \eqref{s-nls}, different from \cite{JWH13}, is obtained via considering the stochastic dispersion. 

%Another objective is to study the numerical approximations of Eq. \eqref{s-nls}. 
Since the analytic solutions of stochastic partial differential equations (PDEs) are rarely known, numerical approximations have become an important tool to investigate the behaviors of the solutions. In order to guarantee the reliability and effectiveness of numerical solutions, especially for longtime simulations, we need numerical methods  preserving the intrinsic properties of the original problems as much as possible.  For stochastic Hamiltonian systems, some important results in constructing stochastic symplectic and multi-symplectic integrators  have been obtained recently (see \cite{CH16, CHZ16, HJZ14, MRT02a, MRT02b} and references therein), which motivates us to develop the corresponding numerical methods for Eq. \eqref{s-nls}.
To inherit stochastic symplectic structure for Eq. \eqref{s-nls}, 
we present two temporal semi-discretizations via the mid-point scheme and the Lie splitting scheme, respectively.
%we consider the mid-point scheme and the Lie splitting scheme in temporal discretization and
These proposed schemes are shown to possess both the charge conservation law and stochastic symplectic structure. 
 Furthermore, thanks to the  unitarity and decay estimates of the stochastic semigroup generated by the linear operator in Eq. \eqref{s-nls}, we deduce mean square order one of these methods for the truncated equation. It then yields 
 order one in probability for the proposed methods approximating the original problem,
i.e., for any fixed $T>0$ and 
$N\in \N^+$ (for details see Theorem \ref{pod}),
\begin{align*}
\lim_{M\to \infty}\mathbb P \Big(\|u(t_n)-u^n\|_{H^s}
\ge M\frac{T}N\Big)=0,\quad 0\leq n\leq N,
\end{align*}
provided $u_0\in H^{s+10}$ for the mid-point scheme (see \eqref{mid-nls}) and 
$u_0\in H^{s+4}$ for the Lie splitting scheme 
(see \eqref{nspl}) with $s>\frac 12$, where $u^n$ is the corresponding numerical solution.
Similar convergence results for Eq. \eqref{nlsf} by other schemes have been  acquired in  \cite{ CHA16, DD06, Liu13} and references therein.
To consider the stochastic multi-symplectic structure, we propose a full-discretization by applying  the spatial centered difference to the temporal mid-point scheme, which is proven to preserve the discrete stochastic multi-symplectic conservation law and charge  conservation law.

\iffalse
{\color{red} We also remark that the temporal Lie splitting method combining with  the spatial centered difference does not preserve the discrete stochastic multi-symplectic conservation law.
}
\fi
The rest of the paper is organized as follows. 
Some preliminaries are given in Section \ref{SSMS}, and we show that Eq. \eqref{s-nls} owns
 the stochastic multi-symplectic conservation law and the stochastic symplectic structure.
In Section \ref{num}, we consider two stochastic symplectic schemes, including the mid-point scheme and  the Lie splitting scheme, and a stochastic multi-symplectic scheme, which preserve the charge conservation law.
Then we prove that their temporal order of convergence is one in probability.
In Section \ref{exp}, we perform numerical experiments to verify our theoretical results and analyze the longtime behaviors of the numerical solutions of Eq. \eqref{s-nls}. 
Finally, we give conclusions in Section \ref{con}.

\section{Stochastic Symplectic and Multi-symplectic Structures}
\label{SSMS}

The existence of a unique solution of Eq. \eqref{s-nls} has been established in \cite{DD10} when $\sigma<2$, and in \cite{DT11} with $\sigma=2$ included.
These results can be extended to a general Sobolev space $H^s$ with $s\geq 1$, i.e., Eq. \eqref{s-nls} has a unique solution $u$ in $\CC(\R_+;H^s)$ if the initial value $u_0\in H^s$, owing to the isometry of stochastic semigroup $ S(t,r)$ in $H^s$.
The mild solution of  Eq. \eqref{s-nls} is 
\begin{align}\label{s-mild}
u(t)=S(t,0)u_0
+\bi \int_0^t S(t,s)|u(s)|^{2\sigma}u(s)\,ds,
\end{align}
where $S$ is the stochastic unitary semigroup generalized by 
$\bi du+\Delta u\circ dW(t)=0$, i.e., for $0\leq s< t \le T<\infty$ and $x\in \R$,
\begin{align}\label{ss}
&S(t,s)\phi(s,x)
:=\frac{1}{\sqrt{4\pi \bi \,(W(t)-W(s))}}
\int_{\R} \exp\left( \bi\,\frac{|x-y|^2}{4(W(t)-W(s))} \right) \phi(s,y)\,dy,
\end{align}
for any $\phi(s,\cdot)\in L^1(\R)$.
Throughout the paper, we assume that $\sigma\leq 2$ and is an integer  so that Eq. \eqref{s-mild} is well-posed and the nonlinear term is sufficiently differentiable.

Just like Eq. \eqref{nlsf}, the charge conservation law is available for  Eq. \eqref{s-nls}. However, on account of white noise dispersion, the evolution of energy is entirely different from the cases in \cite{CH16,DD04,DD06,JWH13}. 

\begin{prop}\label{ccl}
\cite{BDD15}
Eq. \eqref{s-nls} possesses the charge conservation law, i.e.,
\begin{align*}
Q(u(t))= Q(u_0),\quad t\ge0, \quad \text{a.s.}
\end{align*}
where $Q(u(t)):=\int_{\R}|u(t)|^2 \,dx$.
\end{prop}

\begin{rk}\label{ene}
By It\^o formula, it is not difficult to show that the energy 
$$H(u):=\frac 12\int_{\R}|\nabla u|^2\,dx-\frac 1{2\sigma+2}\int_{\R}|u|^{2\sigma+2}\,dx$$
has the following evolution:
\begin{align*}
H(u(t))&=H(u_0)+ \text{Im}\ \int_0^t\int_{\R} \Delta \bar u(s) |u(s)|^{2\sigma}u(s) \,dx ds  \\
&\quad-\text{Im}\ \int_0^t\int_{\R} \Delta \bar u(s) |u(s)|^{2\sigma}u(s) \,dx \circ dW(s),\quad t\geq 0,\quad \text{a.s.} 
\end{align*}
\end{rk}

As we all know, the deterministic NLS equation, i.e., with $W(t)$ in Eq. \eqref{s-nls} replaced by $t$, is an infinite-dimensional Hamiltonian system and a Hamiltonian PDE (see e.g. \cite{HLMZ06} and references therein), which characterizes the geometric invariants of the phase flow and help to construct the numerical schemes for long time computation.
For Eq. \eqref{nlsf}, its stochastic symplectic and multi-symplectic structures have been analyzed in \cite{CH16} and \cite{JWH13}, respectively. In the following, we show that, besides the charge conservation law, Eq. \eqref{s-nls} possesses other geometric invariants, including stochastic symplectic and multi-symplectic structures.

Denote by $p$ and $q$ the real and imaginary parts of $u$, respectively. 
Then Eq. \eqref{s-nls} is equivalent to 
\begin{align*}
&d_tp=-\Delta q\circ dW-|p^2+q^2|^{\sigma}q\,dt,\\
&d_tq=\phantom{-}\Delta p\circ dW+|p^2+q^2|^{\sigma}p\,dt,
\end{align*}
with initial datum $(p(0), q(0))=(p_0,q_0)$.
Set 
\begin{align*}
H_1:=-\frac 12\int_{\R}|\nabla u|^2 \,dx,\quad
H_2:=\frac 1{2\sigma+2}\int_{\R}|u|^{2\sigma+2}\,dx.
\end{align*}
Then the above equation can be rewritten as 
\begin{align}\label{h-nls}
\begin{split}
&d_tp=-\frac {\boldsymbol{\delta} H_1}{\boldsymbol {\delta} q}\circ dW-\frac {\boldsymbol{\delta} H_2}{\boldsymbol{\delta }q}\,dt,\\
&d_tq=\phantom{-}\frac {\boldsymbol{\delta} H_1}{\boldsymbol{\delta} p}\circ dW+\frac {\boldsymbol{\delta} H_2}{\boldsymbol{\delta} p}\,dt,
\end{split}
\end{align}
where $\frac {\boldsymbol{\delta} H_i}{\boldsymbol {\delta} p}$, $\frac {\boldsymbol{\delta} H_i}{\boldsymbol {\delta} q}$ denote the variational derivatives of $H_i$ with respect to $p$ and $q$, $i=1, 2$.
In fact, Eq. \eqref{h-nls} is an infinite-dimensional stochastic Hamiltonian system with Hamiltonians $H_1$ and $H_2$.
One can use the same skill as \cite{CH16} to deduce that Eq. \eqref{h-nls} preserves the stochastic symplectic structure 
\begin{align}\label{sss}
\bar \omega:= \int_{\R} dp\wedge dq \, dx,\quad \text{a.s.}
\end{align}
Alternatively, we prove that Eq. \eqref{s-nls} is a stochastic Hamiltonian PDE, from which we conclude that Eq. \eqref{s-nls} preserves the stochastic symplectic structure \eqref{sss}.
We remark that this argument can also be applied to Eq. \eqref{s-nls} with homogenous Dirichlet or periodic boundary condition.

For skew-symmetric matrices $ M, K$ and a smooth function $S$, the  deterministic $d$-dimensional Hamiltonian PDE 
\begin{align*}
 Mz_t+ Kz_x=\nabla S(z),
\end{align*}
has the multi-symplectic conservation law
\begin{align*}
\partial_t \omega +\partial_x \kappa=0,
\end{align*} 
where $\omega=\frac 12 dz\wedge Mdz$, $\kappa =\frac 12dz\wedge Kdz.$
This law states that the temporal and spatial symplectic structures change locally and synchronously.  

In order to analyze the stochastic multi-symplectic structure, we introduce the state variable $z =(p,q,v,w)^T$, where $v=p_x$ and $w=q_x$. 
Then Eq. \eqref{s-nls} can be transformed into the compact form 
\begin{align}\label{ms-nls}
Md_tz+Kz_x\circ dW=\nabla_z S_1(z)\,dt+\nabla_z S_2(z)\circ dW,
\end{align}
where
\begin{align*}
M=\left(\begin{array}{cccc}0 & -1 & 0 & 0\\1 & 0 & 0 & 0\\0 & 0 & 0 & 0 \\0 & 0 & 0 & 0\end{array}\right),
\quad K=\left(\begin{array}{cccc}0 & 0 & 1 & 0\\0 & 0 & 0 & 1\\-1 & 0 & 0 & 0 \\0 & -1 & 0 & 0\end{array}\right),\\
\end{align*}
and
\begin{align*}
S_1(z)=-\frac 1{2\sigma+2}\left(p^2+q^2\right)^{\sigma+1},\quad S_2(z)=-\frac 12(v^2+w^2).
\end{align*}
The following result shows that Eq. \eqref{ms-nls} preserves the stochastic multi-symplectic conservation law a.s. and thus we call it a stochastic multi-symplectic Hamiltonian system,
which generalizes the scopes of stochastic multi-symplectic Hamiltonian PDE in \cite{JWH13},

\begin{tm}\label{ts-cl}
Eq. \eqref{ms-nls} preserves the stochastic multi-symplectic conservation law locally, i.e.,
\begin{align*}
d_t \omega(t,x) +\partial_x \kappa(t,x)\circ dW(t)=0,\quad \text{a.s.}
\end{align*}
In other words,
\begin{align}\label{es-cl}
\omega(t_1,x)-\omega(t_0,x)=-\int_{t_0}^{t_1}\partial _x \kappa(t,x)\circ dW(t),\quad  \text{a.s.}
\end{align}
\end{tm}

\begin{proof}
Taking the exterior differential on the phase space to Eq. \eqref{ms-nls}, 
we obtain
\begin{align*}
Md(d_tz)+Kdz_x\circ dW=\nabla_{zz} S_1(z) dz dt+\nabla_{zz} S_2(z) dz \circ dW.
\end{align*}
The fact that exterior differential commutes with the stochastic differential and the spatial differential (see e.g. \cite{JWH13}) yields that 
\begin{align*}
Md_t dz+K\partial_x dz\circ dW=\nabla_{zz} S_1(z) dz dt+\nabla_{zz} S_2(z) dz \circ dW.
\end{align*}
Then by wedging $dz$ on the above equation, we have
\begin{align*}
dz\wedge Md_t dz+dz\wedge K\partial_x dz\circ dW=dz\wedge \nabla_{zz} S_1(z) dz dt+dz \wedge \nabla_{zz} S_2(z) dz \circ dW.
\end{align*}
The skew-symmetry of $M$ shows that 
\begin{align*}
\omega(t_1,x)-\omega(t_0,x)
=\int_{t_0}^{t_1} d_t \left(\frac 12 dz\wedge Mdz\right) \,dt
= \int_{t_0}^{t_1} dz\wedge M d_tdz\, dt,
\end{align*}
from which and the symmetry of $\nabla_{zz} S_1(z)$ and $\nabla_{zz} S_2(z)$ we have
\begin{align*}
\omega(t_1,x)-\omega(t_0,x)
&=- \int_{t_0}^{t_1} dz\wedge K\partial_x dz\circ dW
+ \int_{t_0}^{t_1} dz\wedge \nabla_{zz} S_1(z) dz dt\\
&\quad+ \int_{t_0}^{t_1}  dz \wedge \nabla_{zz} S_2(z) dz \circ dW\\
&=- \int_{t_0}^{t_1} dz\wedge K\partial_x dz\circ dW.
\end{align*}
Then we obtain Eq. \eqref{es-cl} by the skew-symmetry of $K$.\qed
\end{proof}

\begin{rk}
\begin{enumerate}

\item
We can repeat the above proof to show that Eq. \eqref{l-nls} is equivalent to the  Hamiltonian PDE
\begin{align*}
 Mz_t+K\epsilon^{-1}m(t\epsilon^{-2})z_x\,dt
=\nabla_z S_1(z)\,dt+{\epsilon}^{-1}m(t\epsilon^{-2})\nabla_z S_2(z)\,dt,
\end{align*}
and that this equation possesses a multi-symplectic conservation law locally.

\item
Under the assumption that the process $t \to \epsilon \int_0^{t\epsilon^{-2}}m(s)ds$ converges in distribution to $W$ in $\CC([0,T])$
as $\epsilon \to 0$, 
the stochastic multi-symplectic conservation law \eqref{es-cl} of Eq. \eqref{ms-nls} is the limit of the multi-symplectic conservation law 
\begin{align*}
\partial_t \omega+
\epsilon^{-1}m(t\epsilon^{-2})\,\partial_x \kappa =0
\end{align*}
of \eqref{l-nls} in distribution. 
\end{enumerate}
\end{rk}
For the high dimensional NLS equation with white noise dispersion, we can deduce its
corresponding stochastic multi-symplectic conservation law by  the above procedures.
The preservation of the stochastic symplectic structure of Eq. \eqref{s-nls} follows immediately from the above theorem. 
% This result is also available for Eq. \eqref{s-nls} with periodic boundary conditions.

\begin{tm}\label{sy-st}
The phase flow of stochastic NLS \eqref{s-nls} preserves the stochastic  symplectic structure \eqref{sss}.
\end{tm}

\begin{proof}
Spatially integrating the multi-symplectic conservation law \eqref{es-cl}, we obtain
\begin{align*}
\int_{\R} \omega(t_1,x)-\omega(t_0,x) \,dx= -\int_{t_0}^{t_1}\int_{\R}\partial _x \kappa(t,x)\,dx\circ dW(t).
\end{align*}
The form of $\kappa$ yields that 
\begin{align*}
\kappa(t,x)=dp \wedge dp_x+dq \wedge dq_x.
\end{align*}
Plugging the above equality into $-\int_{t_0}^{t_1}\int_{\R}\partial _x \kappa(t,x)dx\circ dW(t)$, we get 
\begin{align*}
-\int_{t_0}^{t_1}\int_{\R}\partial _x \kappa(t,x)dx\circ dW(t)
&= \int_{t_0}^{t_1}\int_{\R}\partial_x(dp_{x}\wedge dp+dq_{x}\wedge dq)\,dx\circ dW(t).
\end{align*}
The well-posedness of Eq. \eqref{s-nls} yields that $\kappa(t,\cdot)$ vanishes (see e.g. \cite{CH16}).
Then  we obtain
\begin{align*}
\int_{\R} \omega(t_1,x)\,dx=\int_{\R}\omega(t_0,x) \,dx.
\end{align*}
This completes the proof.
\qed
\end{proof}

\section{Stochastic Symplectic and Multi-symplectic Schemes}
\label{num}

The basic idea for designing the numerical schemes is to make them preserve the 
properties of the real solution as much as possible. In this Section, we give two temporal stochastic symplectic schemes and a fully discrete stochastic multi-symplectic scheme.

Let $0=t_0<t_1<\cdots<t_N=T$ be a equidistant division of $[0,T]$ with time step $\delta t :=\frac TN$ and $\delta W_n:=W(t_{n+1})-W(t_n)$, $n\in\{0,1,\cdots,N-1\}$, be the increments of the Brownian motion.

\subsection{Stochastic Symplectic Schemes}

We start with the stochastic temporal symplectic schemes for Eq. \eqref{s-nls} including the mid-point scheme and the Lie splitting scheme.
%
\iffalse
For stochastic ordinary differential equation, stochastic Runge-Kutta methods play an important role in numerical schemes, which has been widely studied and applied, see \cite{CH16} and references therein.
In order to construct symplectic schemes for \eqref{s-nls}, it is natural to consider stochastic Runge-Kutta
methods. However, constructing symplectic schemes with  the mean square order bigger than 1 for stochastic ordinary differential equation with multiplicative noise is difficult \cite{}. 
\fi

We first consider the mid-point scheme for Eq. \eqref{s-nls}:
\begin{align}\label{mid-nls}
&\bi\,\frac {u^{n+1}-u^n}{\delta t}+\frac {\delta W_n}{\delta t} \Delta u^{n+\frac12}+g(u^{n+\frac12})=0,
\end{align}
where 
\begin{align*}
u^{n+\frac12}=\frac12(u^n+u^{n+1}) \quad \text {and}\quad
g(u^{n+\frac12})=|u^{n+\frac12}|^{2\sigma}u^{n+\frac12}.
\end{align*}
\iffalse
As a special case of symplectic Runge-Kutta methods, it preserves the symplectic structure and the discrete charge conservation law. The following proposition is available for general symplectic Runge-Kutta methods \cite {CH16}.(Besides the symplectic condition, it also needs to satisfy the no-correction condition $\sum a_{ij}b_j=\frac12$.)
\fi
The scheme \eqref{mid-nls} preserves the stochastic symplectic structure 
\begin{align}\label{dsss}
\bar \omega^n:=\int_{\R}dp^{n}\wedge dq^{n}\,dx \quad \text{a.s.}
\end{align}
and the  charge conservation law.

\begin{prop}
The mid-point scheme \eqref {mid-nls} preserves the stochastic symplectic structure \eqref{dsss} and charge conservation law, i.e.,
\begin{align*}
\bar \omega^{n+1}=\bar \omega^{n},\quad 
Q(u^{n+1})=Q(u^n),\quad 0\le n \le N-1, \quad  \text{a.s.}
\end{align*}
\end{prop}

\begin{proof}
Denote $\Phi(p,q):=\frac 1{2\sigma+2}(p^2+q^2)^{\sigma+1}$.
Then Eq. \eqref {mid-nls} can be rewritten as 
\begin{align*}
&p^{n+1}=p^{n}-\Delta q^{n+\frac12}\,\delta W_n-\frac {\partial {\Phi}}{\partial q}(p^{n+\frac12}, q^{n+\frac12})\,\delta t, \\
&q^{n+1}=q^{n}+\Delta p^{n+\frac12}\,\delta W_n+\frac {\partial {\Phi}}{\partial p}(p^{n+\frac12}, q^{n+\frac12})\,\delta t.
\end{align*}
To lighten the notations, we omit the variable $(p^{n+\frac12}, q^{n+\frac12})$ of $\Phi$.
Taking the differential on the phase space, we obtain 
\begin{align*}
&d p^{n+1}=d p^{n}-\Delta dq^{n+\frac12}\,\delta W_n-\frac{\partial^2 \Phi}{\partial q \partial p}dp^{n+\frac12} \,\delta t - \frac{\partial^2 \Phi}{\partial  q^2} dq^{n+\frac 12} \,\delta t, \\
&d q^{n+1}=d q^{n}+\Delta dp^{n+\frac12}\,\delta W_n+\frac{\partial^2 \Phi}{\partial  p^2}dp^{n+\frac12}\,\delta t+\frac{\partial^2 \Phi}{\partial p \partial q}dq^{n+\frac 12} \,\delta t.
\end{align*}
Then
\begin{align*}
&dp^{n+1}\wedge dq^{n+1} \\
&=dp^n \wedge dq^n +dp^n \wedge \left( \Delta dp^{n+\frac12}\,\delta W_n
+\frac{\partial^2 \Phi}{\partial  p^2}dp^{n+\frac12}\,\delta t+\frac{\partial^2 \Phi}{\partial p \partial q}dq^{n+\frac 12} \,\delta t \right)\\
&\quad +\left( -\Delta dq^{n+\frac12}\,\delta W_n-\frac{\partial^2 \Phi}{\partial q \partial p}dp^{n+\frac12}\, \delta t - \frac{\partial^2 \Phi}{\partial  q^2} dq^{n+\frac 12}\, \delta t \right)\wedge dq^n \\
&\quad+\left( -\Delta dq^{n+\frac12}\,\delta W_n-\frac{\partial^2 \Phi}{\partial q \partial p}dp^{n+\frac12}\, \delta t - \frac{\partial^2 \Phi}{\partial  q^2} dq^{n+\frac 12} \,\delta t \right)\\
&\qquad \wedge \left( \Delta dp^{n+\frac12}\,\delta W_n
+\frac{\partial^2 \Phi}{\partial  p^2}dp^{n+\frac12}\,\delta t+\frac{\partial^2 \Phi}{\partial p \partial q}dq^{n+\frac 12} \,\delta t \right).
\end{align*}
Substituting 
\begin{align*}
&d p^n=d p^{n+1}+\frac 12\Delta dq^{n+\frac12}\,\delta W_n+\frac 12\frac{\partial^2 \Phi}{\partial q \partial p}dp^{n+\frac12}\, \delta t +\frac 12\frac{\partial^2 \Phi}{\partial  q^2} dq^{n+\frac 12} \,\delta t, \\
&d q^n=d q^{n+1}-\frac 12\Delta dp^{n+\frac12}\,\delta W_n-\frac 12\frac{\partial^2 \Phi}{\partial  p^2}dp^{n+\frac12}\,\delta t-\frac 12\frac{\partial^2 \Phi}{\partial p \partial q}dq^{n+\frac 12}\, \delta t
\end{align*}
into the above equality and integrating in $\R$, we have
$\bar \omega^{n+1}=\bar \omega^{n}$ a.s.

Multiplying \eqref {mid-nls} by the complex conjugate $\bar u^{n+\frac12}$ of $u^{n+\frac12}$ on both sides, integrating in $\R$ and 
taking the imaginary part, we obtain the
 preservation of the charge conservation law for this scheme.
 \qed
\end{proof}

In the rest of the section, we assume that $s> \frac 12$.
Now we investigate the  convergence rate of the mid-point scheme \eqref{mid-nls}. To deal with the power law of the nonlinear term, we introduce  a cut-off function $\theta \in C^{\infty}(\R)$  such that supp $\theta \in [0,2]$ and $\theta=1$ on $[0,1]$, and denote $\theta_R(u)=\theta(\frac {\|u\|_{H^s}}{R})$ (see also \cite{BDD15,DD04}).
Then we consider the truncated equation 
\begin{align}\label{ts-nls}
du_R=\bi \,\Delta u_R \circ dW + \bi\, \theta_R(u_R)g(u_R)\,dt.
\end{align}
The mild form of the corresponding mid-point scheme of Eq. \eqref{ts-nls} is
\begin{align}\label{ts-mid}
&u_R^{n+1}=S_{n,\delta t}u_R^n +\bi \,\delta t\,T_{n,\delta t}\,\theta_R(u_R^{n+\frac12})\,g(u_R^{n+\frac12}),
\end{align}
where 
\begin{align}\label{op}
S_{n,\delta t}=\left(I-\bi \frac {\delta W_n}{2} \Delta \right)^{-1}\left(I+\bi \frac {\delta W_n}2\Delta \right) \quad\text{and}\quad
T_{n,\delta t}=\left(I-\bi \frac {\delta W_n}{2}  \Delta \right)^{-1}.
\end{align}
The scheme \eqref{ts-mid} is well-defined and $\{u_R^n\}_{n=0}^N$ are uniformly bounded provided that the nonlinear term $\theta_R\cdot g$ is Lipschitz continuous.
This Lipschitz continuity can be guaranteed by the following result.
It can be proved directly by Proposition 2.2 in \cite{BDD15} because $H^s$($s>\frac 12$) is an algebra.

\begin{lm}\label{lip}
Let $s>\frac 12$. 
Then $g_R:=\theta_R\cdot g\in \text{Lip}(H^s)\cap \CC_b^2(H^s;H^s)$, i.e., for any $u,v,w\in H^s$, there exists a constant $C(\sigma,R)$ such that 
\begin{align*}
\|g_R^{}(u)-g_R^{}(v)\|_{H^s} 
&\le C(\sigma,R)\|u-v\|_{H^s},\\
\|Dg_R^{}(u)(v)\|_{H^s} 
&\le C(\sigma,R)\|v\|_{H^s},\\
\|D^2g_R^{}(u)(v,w)\|_{H^s} 
&\le C(\sigma,R)\|v\|_{H^s}\|w\|_{H^s}.
\end{align*}
\end{lm}

Assume that $u_0\in H^{s+10}$.
Since the nonlinear term of Eq. \eqref{ts-nls} satisfies the global Lipschitz condition established in the above Lemma, the solution $u_R$ has paths in $H^{s+10}_x$ a.s. 
Moreover, the Gronwall inequality yields the moments' uniform boundedness of $u_R$, i.e., there is a  constant $C=C(T,R,\|u_0\|_{H^{s+10}})$ such that for any $p\ge 2$,
\begin{align*}
\E\left( \sup_{t\le T}\|u_R(t)\|_{H^{s+10}}^p\right) \le C.\end{align*}
It turns out in the following result that the scheme \eqref{ts-mid} is convergent with mean square order one. 
The key of the proof  lies on the mild form of the solution and the unitarity of both $S(t,r)$ and $S_{n,\delta t}$. 

\begin{prop}\label{ts-order}
Let $u_0 \in H^{s+10}.$ The scheme \eqref{ts-mid} has mean square order one, i.e., there exists 
a constant $C=C(T,R,\|u_0\|_{H^{s+10}})$ such that 
\begin{align}\label{ts-od}
\sup_{0\le n \le N}\sqrt{\E\left(\|u_R^n-u_R(t_n)\|_{H^s}^2 \right)}\le C\delta t.
\end{align}
\end{prop}

\begin{proof}
For simplicity, we omit the dependence $R$ of $u^n$ and $u$.
Denote by $e_n=u^n-u(t_n)$ the local error, $0\le n \le N$. 
Subtracting the mild formulation 
$$u(t_{n+1})=S(t_{n+1},t_n)u(t_n)+\bi \int_{t_n}^{t_{n+1}}S(t_{n+1},r)g_R^{}(u(r))\,dr$$
of \eqref{ts-nls} from the scheme \eqref{ts-mid},
we have $e_{n+1}=\sum\limits_{j=1}^5 e_{n+1,j}$
with 
\begin{align*}
&e_{n+1,1}=S_{n,\delta t}e_n,\\
&e_{n+1,2}=\big(S_{n,\delta t}-S(t_{n+1},t_n)\big)u(t_n),\\
&e_{n+1,3}=\bi \int_{t_n}^{t_{n+1}}S(t_{n+1},r)\big(g_R^{}(u(t_n))-g_R^{}(u(r))\big)dr,\\
&e_{n+1,4}=\bi \big(\delta t\, T_{n,\delta t}-\int_{t_n}^{t_{n+1}}S(t_{n+1},r)g_R^{}(u(t_n))dr\big),
\end{align*}
and
\begin{align*}
e_{n+1,5}=\bi\,\delta t \, T_{n,\delta t} \, \big(g_R^{}(u^{n+\frac 12})-g_R^{}({u(t_n)})\big).
\end{align*}
The Cauchy-Schwarz inequality yields
\begin{align*}
\E(\|e_{n+1}\|_{H^s}^2)
&\le \E(\|e_{n+1,1}\|_{H^s}^2)+
3\sum_{j=2}^4 \E(\| e_{n+1,j}\|_{H^s}^2)\\
&\quad+2 \left|\sum_{j=2}^4 \E(e_{n+1,1},\,e_{n+1,j})_{H^s}\right|+2 \left|\sum_{j=1}^5  \E(e_{n+1,j},\,e_{n+1,5})_{H^s}\right|.
\end{align*}
It is shown  in \cite[Proposition 2.6]{BDD15} that
\begin{align*}
\E(\|e_{n+1}\|_{H^s}^2 ) \le (1+\delta t\,K_1) \E(\|e_n\|_{H^s}^2)+(\delta t)^3 K_2 +2 \left| \sum_{j=2}^5 \E(e_{n+1,j},\,e_{n+1,5})_{H^s} \right|.
\end{align*}
It remains to deal with the last term. 
Since $T_{n,\delta t}$ and $g_R^{}$ are bounded in $H^s$,
\begin{align*}
\E(\|e_{n+1,5}\|_{H^s}^2)
&\le C_R (\delta t)^2 \E (\|u^{n+\frac12}-u(t_n)\|_{H^s}^2) \\\nonumber
&\le C_R (\delta t)^2 \E \left(\left\|\frac {I+S_{n,\delta t}}2 u^n-u({t_n})+\bi\,\delta t\,T_{n,\delta t}\,g_R^{}(u^{n+\frac 12})\right\|_{H^s}^2\right).
\end{align*}
By Cauchy-Schwarz inequality and the boundedness of the solution, we have 
\begin{align*}
&\E \left( \left\|\frac {I+S_{n,\delta t}}2 u^n-u({t_n})+\bi\,\delta t\,T_{n,\delta t}\,g_R^{}(u^{n+\frac 12}) \right\|_{H^s}^2 \right)  \\ 
&\lesssim \E\left(\left\|\frac {I+S_{n,\delta t}}2\big( u({t_n})-u^n\big)\right\|_{H^s}^2\right)+\E\left(\left\|\left(\frac {I+S_{n,\delta t}}2 -I\right)u(t_n)\right\|_{H^s}^2\right)  \\
&\quad +(\delta t)^2\E\left(\|u^{n+\frac 12}\|_{H^s}^2\right) \\
&\lesssim \E \|e_n\|_{H^s}^2+\E \left(\left\|\left(\frac {I+S_{n,\delta t}}2 -I\right)u(t_n)\right\|_{H^s}^2\right)+(\delta t)^2 \left( I+\E\left(\|u_0\|_{H^s}^2\right) \right).
\end{align*}
Here and what follows, $A \lesssim B$ denotes $A\leq CB$ for a positive constant $C$ depending only on $R$ and $T$.
The Parseval theorem, combined with the independence of $S_{n,\delta t}$ and $ u({t_n})$, implies
\begin{align*}
&\E\left(\left\|\left(\frac {I+S_{n,\delta t}}2 -I\right)u({t_n})\right\|_{H^s}^2\right)\\
&= \int_{\R}\E\left(\left|\frac {1-\exp(-2\bi \arctan (|\xi|^2 \frac {\delta W_n}2)) }2\right|^2 |\mathcal F (u({t_n}))|^2(\xi)\right)(1+|\xi|^2)^s\,d\xi \\
&\lesssim \delta t\, \E(\|u({t_n})\|_{H^{s+2}}^2) \lesssim \delta t \,(1+\E(\|u_0\|_{H^{s+2}}^2)).
\end{align*}
Here and what after $\mathcal F$ denotes the Fourier transform.
Therefore,
\begin{align*}
\E(\|e_{n+1,5}\|_{H^s}^2)
\lesssim (\delta t)^2 E(\left \|e_n \right \|_{H^s}^2)+\left(1+\E(\|u_0\|_{H^s}^2)+\E(\|u_0\|_{H^{s+2}}^2)\right) (\delta t)^3. 
\end{align*}
For the term $\E (e_{n+1,1},\,e_{n+1,5})_{H^s}$, we have
\begin{align*}
&|\E (e_{n+1,1},\,e_{n+1,5})_{H^s}|\\
%&\le \left| \E\left(S_{n,\delta t}e_n,\,\bi\,\delta t \,T_{n,\delta t}\,(g_R^{}(u^{n+\frac 12})-g_R^{}({u(t_n)}))\right)_{H^s}\right|\\
&\le \left|\E\left(e_n,\,\bi\,\delta t\, S_{n,\delta t}^{*}\,T_{n,\delta t}\,\big(g_R^{}( u^n)-g_R^{}({u(t_n)})\big)\right)_{H^s} \right|\\
&\quad + \left|\E\left(e_n,\,\bi\,\delta t\, S_{n,\delta t}^{*}\,T_{n,\delta t}\,\big(g_R^{}(u^{n+\frac 12})-g_R^{}( u^n)\big)\right)_{H^s} \right|
=: A_1+A_2.
\end{align*}
By the boundedness of $T_{n,\delta t}$ and the Lipschitz continuity of $g_R^{}$,  it is not difficult to show that
\begin{align*}
A_1
\le \delta t\,E(\|e_n \|^2_{H^s}).
\end{align*}
To deal with $A_2$, we give an estimation
of $u^{n+\frac12}-u^n$. Similar arguments imply that 
\begin{align*}
|\E(\|u^{n+\frac12}-u^n\|^2_{H^s})|
&\lesssim \E\left( \left \| \left(\frac {S_{n,\delta t}-I}2\right) u^n\right \|^2_{H^s}\right)
+(\delta t)^2 \left( 1+\E(\|u_0\|_{H^s}^2) \right).
\end{align*}
Fourier transform and the independence of $S_{n,\delta t}$ and $ u^n$ yield that 
\begin{align*}
|\E(\|u^{n+\frac12}-u^n\|^2_{H^s})|
&\lesssim \ \delta t \,\E(\| u^n\|^2_{H^{s+2}})
+(\delta t)^2 \left( 1+\E(\|u_0\|_{H^s}^2) \right)\\
&\lesssim \ \delta t\left( 1+\E(\|u_0\|_{H^s}^2)
+\E \left( \| u^n\|^2_{H^{s+2}} \right) \right).
\end{align*}
Now we turn to the estimate of $A_2$.
Again by the Taylor expansion of $g_R^{}$, there exists an $\eta$
depending on $u^n$ and $u(t_n)$ such that 
\begin{align*}
A_2
\le & \left|\E\left(e_n,\,\bi\,\delta tS_{n,\delta t}^{*}\,T_{n,\delta t}\,Dg_R^{}( u^n)
\big(u^{n+\frac 12}- u^n\big)\right)_{H^s} \right|\\
&+ \frac12 \left| \E\left(e_n,\,\bi\,\delta tS_{n,\delta t}^{*}\,T_{n,\delta t}\,D^2g_R^{}(\eta)
\big(u^{n+\frac 12}- u^n,u^{n+\frac 12}- u^n\big)\right)_{H^s} \right|\\
=:& A_{21} +A_{22}.
\end{align*}
The term $A_{21}$ can be estimated as  
\begin{align*}
A_{21}&\le \left|\E\left(e_n,\,\bi\,\delta t S_{n,\delta t}^{*}\,T_{n,\delta t}\,Dg_R^{}( u^n)
\frac 12\big(S_{n,\delta t}-I\big)u_n\right)_{H^s} \right|\\
&\quad+ \left|\E\left(e_n,\,\bi\,\delta tS_{n,\delta t}^{*}\,T_{n,\delta t}\,Dg_R^{}( u^n)
\frac 12 \int_{t_n}^{t_{n+1}}T_{n,\delta t}g_R^{}(u_{n+\frac 12})\,dr \right)_{H^s} \right|\\
&\le \left|\E\left(e_n,\,\bi\,\delta t S_{n,\delta t}^{*}\,T_{n,\delta t}\,Dg_R^{}( u^n)
\frac 12\big(S_{n,\delta t}-I\big)u_n\right)_{H^s} \right|\\
&\quad+C(\delta t)^2 \E\left(\|e_n\|_{H^s}^2\right)^{\frac 12}(1+\E\left(\|u_0\|_{H^{s}}^4\right)^\frac 12).
%& \le \left|\E\left(e_n,\,\bi\,\delta t \left((S_{n,\delta t}^{*}\,T_{n,\delta t}-I)\,Dg_R^{}( u^n)\right )
%\frac 12(S_{n,\delta t}-I)u_n\right)_{H^s} \right|\\
%&\quad +\left|\E\left(e_n,\,\bi \,\delta t\left(Dg_R^{}( u^n)\right )
%\frac 12(S_{n,\delta t}-I)u_n\right)_{H^s} \right|.\\
%&\quad+C(\delta t)^2 \E\left(\|e_n\|_{H^s}^2\right)^{\frac 12}(1+\E\left(\|u_0\|_{H^{s}}^4\right)^\frac 12),
\end{align*}
The independence of $u(t_n)$ and $\delta W_n$, together with the boundedness of $Dg_R^{}$ and Plancherel theorem, yields that 
\begin{align*}
&\left|\E\left(e_n,\,\bi\,\delta tS_{n,\delta t}^{*}\,T_{n,\delta t}\,Dg_R^{}( u^n)
\frac 12\big(S_{n,\delta t}-I\big)u_n\right)_{H^s} \right|\\
&\lesssim (\delta t)^2 \E(\|e_n\|_{H^s}^2)^\frac12 (1+ \E(\|u_0\|_{H^{s+4}}^4)^\frac12). 
\end{align*}
The above estimates yield that 
\begin{align*}
A_{21}
\lesssim (\delta t)^2 \E(\|e_n\|_{H^s}^2)^\frac12 (1+ \E(\|u_0\|_{H^{s+4}}^4)^\frac12).
\end{align*}
From the estimation of $|\E(\|u^{n+\frac12}-u^n\|^2_{H^s})|$ and Cauchy-Schwarz inequality, we obtain
\begin{align*}
A_{22}
\lesssim (\delta t)^2 \E(\|e_n\|_{H^s}^2)^\frac12 (1+ \E(\|u_0\|_{H^{s+2}}^4)^{\frac 12} .
\end{align*} 
The above estimations on $A_1$ and $A_2$ imply 
\begin{align*}
|\E (e_{n+1,1},\,e_{n+1,5})_{H^s}|
\lesssim \delta t\,\E(\|e_n\|^2_{H^s})
+(\delta t)^3 \left( 1+ \E(\|u_0\|_{H^{s+2}}^4)+ \E(\|u_0\|_{H^{s+4}}^4) \right).
\end{align*}
Finally, all terms of the form $|\E (e_{n+1,j},\,e_{n+1,5})_{H^s}|$, $j=2,3,4$ are bounded by Cauchy-Schwarz inequality.
There exist positive constants $\varepsilon,K_1$ and $K_2$ such that for any 
$\delta t\le \varepsilon$ and $n=0,1\cdots,N-1$,
\begin{align*}
\E(\|e_{n+1}\|_{H^s}^2 ) \le (1+\delta t \,K_1) \E(\|e_n\|_{H^s}^2)+(\delta t)^3 K_2.
\end{align*}
We complete the proof by Gronwall inequality.
\qed
\end{proof}

Since the midpoint scheme is implicit, we expect a more efficient one, which
 uses splitting techniques and 
explicit formulas to solve the subproblems containing the nonlinear terms.
For  Schr\"odinger-type equations, Liu \cite{Liu13b,Liu13} generalizes them to stochastic NLS equations.
For the case of NLS equation \eqref{l-nls} with random dispersion,  \cite{Mar06} has analyzed a splitting scheme.
In this part, we consider the Lie splitting scheme for Eq. \eqref{s-nls}, and prove rigorously the error estimate.
\iffalse
{\color{red}
Due to the efficiency of splitting schemes, there are some work concerning splitting schemes for Schr\"odinger-type equations.
Liu \cite{Liu13, Liu13b} generalizes them to stochastic NLS equations.
For the case of NLS equation \eqref{l-nls} with random dispersion,  \cite{Mar06} has analyzed a splitting scheme.
In this part, we consider the Lie splitting scheme for Eq. \eqref{s-nls}, and prove rigorously the error estimate.
}
\fi
Note that both the nonlinear equation
\begin{align}\label{spl1}
du=\bi\, |u|^{2\sigma}udt,\quad u(0)=u_0,
\end{align}
and the linear equation
\begin{align}\label{spl2}
du=\bi \,\Delta u\circ dW,\quad u(0)=u_0,
\end{align}
can be solved explicitly (see e.g. \cite {Liu13}).
This fact leads to the following Lie splitting scheme:
\begin{align}\label{nspl}
&\tilde u^n=\exp(\bi\,\delta t \,|u^n|^{2\sigma})u^n,\quad 
u^{n+1}=\exp(\bi \,\delta W_n \Delta)\tilde u^n.
\end{align}

\begin{prop}\label{nspl-sc}
The splitting scheme \eqref{nspl} preserves the stochastic symplectic structure \eqref{dsss} and  charge conservation law, i.e.,
\begin{align*}
\bar \omega^{n+1}=\bar \omega^{n},\quad 
Q(u^{n+1})=Q(u^n),\quad 0\le n\le N-1,\quad  \text{a.s.}
\end{align*}
\end{prop}

\begin{proof}
We rewrite Eq. \eqref{spl1} and Eq. \eqref{spl2} as
\begin{align*}
dp=-|p^2+q^2|^{\sigma}q\,dt,\quad
&dq=|p^2+q^2|^{\sigma}p\,dt,\\
dp=-\Delta q\circ dW,\quad
&dq=\Delta p\circ dW.
\end{align*}
Due to the form of Eq. \eqref{h-nls}, we have that the above two equations are both infinite-dimensional Hamiltonian system
with Hamiltonians $H_2$ and $H_1$, respectively.
It follows that
\begin{align*}
 \int_{\R}  dp^{n+1}\wedge dq^{n+1}\,dx
 =\int_{\R} d\tilde p^{n}\wedge d \tilde q^n \,dx
 =\int_{\R} dp^n\wedge dq^n\,dx.
\end{align*}
The charge conservation law follows immediately from the explicit formulations \eqref{nspl} and the unitarity of $\exp(\bi \delta t |\cdot|^2)$ and $\exp(i \delta W_n\Delta)$.
\qed
\end{proof}

In order to analyze the convergence order of the splitting scheme, we start with the truncated equation \eqref{ts-nls}, whose splitting scheme becomes 
\begin{align}\label{t-spl}
\tilde u_R^n=\exp(\bi \,\delta t\, \theta_R(u_R^{n})|u_R^{n}|^2)u_R^{n},\quad
u_R^{n+1}=\exp(\bi \,\delta W_n \,\Delta) \tilde u_R^n.
\end{align}
Introduce a right continuous function $\psi_N(t)$ such that $\psi_N(0)=u_0$ and on any interval $[t_{n},t_{n+1}]$,
\begin{align*}
\psi_{N}(t)=
\begin{cases}
\exp(\bi \,\delta t \,\theta_R(\psi_{N}(t_n))|\psi_{N}(t_n))|^{2\sigma})\psi_{N}(t_n),\quad &t\in [t_n,t_{n+1}),\\
\exp(\bi \,\delta W_n\,\Delta)\lim_{t\to t_{n+1}}\psi_{N}(t),\quad &t=t_{n+1}.
\end{cases}
\end{align*}
Then $\psi_{N}(t_{n+1})=u_R^{n+1}$ and
$d\psi_{N}=\bi\, g_R^{}(\psi_{N})dt$ for $t\in[t_{n},t_{n+1})$,
i.e.,
\begin{align}\label{tc-ddif}
\psi_N(t)
&=u_R^n+\bi \int_{t_n}^{t}g_R^{}(\psi_N(s))\,ds,\quad t\in[t_{n},t_{n+1}).
\end{align}
\iffalse
Recall the mild solution for \eqref{ts-nls}
\begin{align}\label{te-eq}
u_R(t_{n+1})=S(t_{n+1},t_n)u_R(t_n)+\bi \int_{t_n}^{t_{n+1}}S(t_{n+1},r)g_R^{}(u(r))dr\ .
\end{align}
Compared with the splitting scheme 
\begin{align}\label{te-nu}
u_R^{n+1}=\psi_{N}(t_{n+1})=S(t_{n+1},t_n)u_R^{n}+\bi \int_{t_n}^{t_{n+1}}S(t_{n+1},t_n)g_R^{}((\psi_N(x,s)))ds,
\end{align}
\fi

Since $g_R$ is Lipschitz, it is not difficult to show that
 $\psi_{N}$ and $u_R^n$ are uniformly bounded.
Based on this property, we prove that the scheme \eqref{t-spl} has mean square convergence order one.
\begin{prop}\label{od-spl}
Assume that $u_0\in H^{s+4}$. The splitting scheme \eqref{t-spl} has mean square order one, i.e.,
there is a constant $C=C(T,R,\|u_0\|_{H^{s+4}})$ such that 
\begin{align*}
\sup_{0\le n \le N}\E\left(\|u_R^n-u_R(t_n)\|_{H^s}^2 \right)\le C(\delta t)^2.
\end{align*}
\end{prop}

\begin{proof}
As in Proposition \ref{ts-order}, we omit the dependence of $R$ and only need to prove  
\begin{align}\label{k-od-spl}
\E(\|e_{n+1}\|_{H^s}^2)\le(1+\delta t \,K_1)\E(\|e_n\|_{H^s}^2)+(\delta t)^3K_2,
\end{align}
for small enough $\delta t$ and some positive constants $K_1$ and $K_2$.
The error $e_{n+1}=u^{n+1}-u(t_{n+1})$ can be divided into 
$$\qquad e_{n+1}=\sum_{j=1}^3 e'_{n+1,j},$$
where 
\begin{align*}
&e'_{n+1,1}=S(t_{n+1},t_{n})e_n,\\
&e'_{n+1,2}=\int_{t_n}^{t_{n+1}}S(t_{n+1},r)\big(g_R^{}(\psi_N(r))-g_R^{}(u(r))\big)\,dr,\\
\end{align*}
and
$$e'_{n+1,3}=\int_{t_n}^{t_{n+1}}\big(S(t_{n+1},t_n)-S(t_{n+1},r)\big)g_R^{}(\psi_N(r))\,dr.$$
Then
\begin{align}\label{od6}
\E(\|e'_{n+1}\|_{H^s}^2)&=\sum_{j=1}^3\E(\|e'_{n+1,j}\|_{H^s}^2)+2\E(e'_{n+1,1},e'_{n+1,2})_{H^s}\\\nonumber
&\quad +2\E(e'_{n+1,1},e'_{n+1,3})_{H^s}+2\E(e'_{n+1,2},e'_{n+1,3})_{H^s}.
\end{align}
Next, we will estimate separately the terms appearing in the previous equality.

It is obvious that 
\begin{align}\label{t1}
\E(\|e'_{n+1,1}\|_{H^s}^2)\le \E (\|e_n\|_{H^s}^2).
\end{align}
Applying H\"older inequality and the Lipschitz continuity of $g$, we obtain
\begin{align*}
&\E(\|e'_{n+1,2}\|_{H^s}^2)\\
&\lesssim \delta t\,\int_{t_n}^{t_{n+1}}\E(\|g_R^{}(u(r))-g_R^{}(\psi_N(r))\|_{H^s}^2)\,dr\\
&\lesssim \delta t\,\int_{t_n}^{t_{n+1}}\E \left( \left\|\big(S(r,t_n)-I\big)u(t_n)-e_n
+\bi \int_{t_n}^rS(r,\sigma)\big(g_R^{}(u(\sigma))-g_R^{}(\psi_N(\sigma))\big)d\sigma \right\|_{H^s}^2 \right)dr.
\end{align*}
Then again by the H\"older inequality and the boundedness of $u(t)$, $\psi_N(t)$ and $Dg$,
\begin{align*}
&\int_{t_n}^{t_{n+1}}\E \left( \left\|\big(S(r,t_n)-I\big)u(t_n)-e_n
+\bi\int_{t_n}^r S(r,\sigma)\big(g_R^{}(u(\sigma))-g_R^{}(\psi_N(\sigma))\big)d\sigma \right\|_{H^s}^2 \right)\,dr  \\
&\lesssim \int_{t_n}^{t_{n+1}}\E(\|\big(S(r,t_n)-I\big)u(t_n)\|_{H^s}^2)\,dr+\int_{t_n}^{t_{n+1}}\E(\|e_n\|_{H^s}^2)\,dr\\
&\quad+\int_{t_n}^{t_{n+1}}\E\left( \left\| \int_{t_n}^r S(r,\sigma)\big(g_R^{}(u(\sigma))-g_R^{}(\psi_N(\sigma))\big)d\sigma \right\|_{H^s}^2\right)\,dr\\
&\lesssim(\delta t)^2(1+\E(\|u_0\|_{H^{s+2}}^2))+\delta t\,E(\|e_n\|_{H^s}^2)+(\delta t)^3(1+\E(\|u_0\|_{H^s}^2)),
\end{align*}
from which we have
\begin{align}\label{t2}
\E(\|e'_{n+1,2}\|_{H^s}^2)
\lesssim (\delta t)^2\E(\|e_n\|_{H^s}^2)+(\delta t)^3(1+\E(\|u_0\|_{H^{s+2}}^2)).
\end{align}
For the term $\E(\|e'_{n+1,3}\|_{H^s}^2)$, we have
\begin{align*}
\E(\|e'_{n+1,3}\|_{H^s}^2)
&\le 2\E \left (\left \|\int_{t_n}^{t_{n+1}}\big(S(t_{n+1},r)-I\big)g_R^{}(\psi_N(r))\,dr \right \|_{H^s}^2 \right)\\
&\quad+2\E \left(\left \|\int_{t_n}^{t_{n+1}}\big(S(t_{n+1},t_n)-I\big)g_R^{}(\psi_N(r))\,dr \right \|_{H^s}^2 \right).
\end{align*}
We only deal with the first term in the above inequality and similar arguments are available for the last term.
By Young inequality and H\"older inequality, we have
\begin{align*}
&\E(\|\int_{t_n}^{t_{n+1}}\big(S(t_{n+1},r)-I\big)g_R^{}(\psi_N(r))\,dr\|_{H^s}^2)\\
&\lesssim \delta t\int_{t_n}^{t_{n+1}} \E\left( \left\| \big(S(t_{n+1},r)-I\big)g_R^{}(\psi_N(t_n)) \right\|_{H^s}^2 \right)\,dr\\
&\quad +\E\left (\left \|\int_{t_n}^{t_{n+1}}\big(S(t_{n+1},r)-I\big)\big(g_R^{}(\psi_N(r)-g_R^{}(\psi_N(t_n)))\big)\,dr\right \|_{H^s}^2 \right ).
\end{align*}
By Lemma 2.10 in \cite{BDD15},
$$\E(\|\big(S(r,t_n)-I\big)g_R^{}(\psi_N(r))\|_{H^s}^2)\le C\,\delta t\, \E(\|g_R^{}(\psi_N(r))\|_{H^{s+2}}^2),$$
which implies that 
$$\int_{t_n}^{t_{n+1}} \E\left( \left\| \big(S(t_{n+1},r)-I\big)g_R^{}(\psi_N(t_n)) \right\|_{H^s}^2 \right)dr
\lesssim (\delta t)^2 \left(1+\E \left( \left \|u_0 \right \|_{H^{s+2}}^2\right)\right).$$
The boundedness of $Dg$ and Taylor expansion yield the existence of $\eta'$ such that 
\begin{align*}
&\E\left (\left \|\int_{t_n}^{t_{n+1}}(S(t_{n+1},r)-I)\big(g_R^{}(\psi_N(r)-g(\psi_N(t_n)))\big)\,dr\right \|_{H^s}^2 \right )\\
&\lesssim \E \left ( \left \|\int_{t_n}^{t_{n+1}}((S(t_{n+1},r)-I)Dg_R^{}(\eta')) \left(\bi\int_{t_n}^{r}g_R^{}(\psi_N(\sigma))\,d\sigma dr\right)  \right \|_{H^s}^2 \right)\\
&\lesssim (\delta t)^2 \int_{t_n}^{t_{n+1}}\int_{t_n}^{r} \E(\|g_R^{}(\psi_N(\sigma)\|_{H^s}^2)\,d\sigma dr
\lesssim (\delta t)^4 \left( 1+\E\left( \|u_0\|_{H^s}^2 \right) \right).
\end{align*}
The above inequalities imply that 
\begin{align}\label{t3}
\E(\|e'_{n+1,3}\|_{H^s}^2)
\lesssim (\delta t)^3\left(1+\E \left( \left \|u_0 \right \|_{H^{s+2}}^2\right)+\E\left( \|u_0\|_{H^s}^2 \right)\right).
\end{align}
Next we turn to the correlated terms in \eqref{od6}.
Cauchy-Schwarz inequality combined with the estimations on $\E(\|e'_{n+1,2}\|_{H^s}^2)$ and $\E(\|e'_{n+1,3}\|_{H^s}^2)$ lead to
\begin{align}\label{t4}
&|\E(e'_{n+1,2},\,e'_{n+1,3})_{H^s}|  \nonumber  \\
&\lesssim \E(\|e'_{n+1,2}\|_{H^s}^2)^{\frac12}\E(\|e'_{n+1,3}\|_{H^s}^2)^{\frac12}\nonumber  \\
&\lesssim (\delta t)^{\frac 52}E(\|e_n\|_{H^s}^2)^{\frac 12}
\left(1+\E\left( \|u_0\|_{H^{s+2}}^2 \right)\right)^{\frac 12}
+(\delta t)^3 \left(1+\E\left( \|u_0\|_{H^{s+2}}^2 \right) \right)  \nonumber  \\
&\lesssim \delta t\,E(\|e_n\|_{H^s}^2)
+(\delta t)^3 \left(1+\E\left( \|u_0\|_{H^{s+2}}^2 \right) \right).
\end{align}
For the term $\E(e'_{n+1,1},e'_{n+1,2})_{H^s}$, we have
\begin{align*}
&\E(e'_{n+1,1},e'_{n+1,2})_{H^s} \\
&=\E\left (e_{n},\,\int_{t_n}^{t_{n+1}}\big(g_R^{}(u(t_n))-g_R^{}(u(r))\big)dr \right )_{H^s}\\
&\quad+\E\left (e_{n},\,\int_{t_n}^{t_{n+1}}\big(S(t_{n},r)-I\big)\big(g_R^{}(u(t_n))-g_R^{}(u(r))\big)dr \right )_{H^s}\\
&\quad+\E \left (e_{n},\,\int_{t_n}^{t_{n+1}}S(t_n,r)\big(g_R^{}(u^n))-g_R^{}(u(t_n))\big)dr \right)_{H^s}\\
&\quad+\E \left(e_{n},\,\int_{t_n}^{t_{n+1}}S(t_n,r)\big(g_R^{}(\psi_N(r))-g_R^{}(u^n)\big)dr \right)_{H^s}
=: B_1+ B_2+B_3+B_4.
\end{align*}
Since $g_R^{}\in \CC_b^2$, by Lemma 2.10 in \cite{BDD15}, the term $B_1$ can be estimated similarly as
\begin{align*}
B_1 &\le \left|\E \left( e_n,\,\int_{t_n}^{t_{n+1}}Dg_R^{}(u(t_n))\big((S(r,t_n)-I)u(t_n)\big)\,dr\right)_{H^s}\right|  \\
& \quad+\left|\E \left(e_n,\,\int_{t_n}^{t_{n+1}} \int_{t_n}^rS(r,\sigma)g_R^{}(u(\sigma))d\sigma)\,dr\right)_{H^s}\right|
+C(\delta t)^2\E(\|e_n\|_{H^s}^2)^{\frac12}\\
&\lesssim (\delta t)^{\frac 12}\E(\|e_n\|_{H^s}^2)^{\frac12}
\left(\int_{t_n}^{t_{n+1}}\E(\|\E\big(S(r,t_n)-I\big) u(t_n)\|_{H^s}^2) \,dr\right) ^{\frac 12}+(\delta t)^2\E(\|e_n\|_{H^s}^2)^{\frac12}.
\end{align*}
Since $S(r,t_n)$ is independent of $u(t_n)$ and $Dg(u(t_n))$, by Parseval theorem, we have
\begin{align*}
&\E(\|\E(S(r,t_n)-I) u(t_n)\|_{H^s}^2\\
&=\E \left(\int_{\R}|\E(\exp(\bi(\beta(t_n)-\beta(r))|\xi|^2)-1|^2|\hat u(t_n)|^2(1+\xi^2)^s d\xi \right)\\
&\lesssim (\delta t)^2\sup_{t\in [0,T]}\E(\|u(t)\|_{H^{s+4}}^2)
\lesssim (\delta t)^2\E(\|u_0\|_{H^{s+4}}^2).
\end{align*}
Then $B_1$ can be controlled by 
\begin{align*}
B_1
\lesssim \delta t\,\E(\|e_n\|_{H^s}^2)+(\delta t)^3(1+\E(\|u_0\|_{H^{s+4}}^2)).
\end{align*}
Parseval theorem yields that for $\FFF_{t_n}$ measurable function $v$, 
\begin{align*}
\E\left(\|\big(S(t_{n},r)-I\big)v\|_{H^s}^2\right)\le C\,\delta t\, \E\left(\|v\|_{H^{s+2}}^2\right). 
\end{align*}
Applying the above inequality to term $B_2$, combining with $g_R^{}\in \CC_b^2$, we obtain
\begin{align*}
B_2  &\lesssim (\delta t)^2 \E(\|e_n\|_{H^s}^2)^{\frac 12}\sup_{t\in[0,T]}\E \left(\| u(t)\|_{H^{s+4}}^2 \right)^{\frac 12}\\
& \lesssim \delta t\,\E(\|e_n\|_{H^s}^2)+ (\delta t)^3(1+\E(\|u_0\|_{H^{s+4}}^2))
\end{align*} 
Analogously, 
\begin{align*}
B_3 \lesssim \delta t\,\E(\|e_n\|_{H^s}^2),
\end{align*}
and 
\begin{align*}
B_4 
\lesssim \delta t\,\E(\|e_n\|_{H^s}^2)+(\delta t)^3(1+\E(\|u_0\|_{H^{s}}^2)).
\end{align*}
The estimates of $B_1$--$B_4$ imply
\begin{align}\label{t5}
\E(e'_{n+1,1},e'_{n+1,2})_{H^s}\lesssim 
\delta t\,\E(\|e_n\|_{H^s}^2)+ (\delta t)^3(1+\E(\|u_0\|_{H^{s+4}}^2)).
\end{align}
For the term $\E(e'_{n+1,1},e'_{n+1,3})$, we split it as 
\begin{align*}
&|\E(e'_{n+1,1},e'_{n+1,3})|\\
&\le C \left|\E \left (e_n,\,\int_{t_n}^{t_{n+1}}\big(S(t_{n},r)-I)g_R^{}(u^n\big)\,dr\right)_{H^s}\right|\\
&\quad+C \left|\E \left(e_n,\,\int_{t_n}^{t_{n+1}}\big(S(t_{n},r)-I\big)\big(g_R^{}(\psi_N(r))-g_R^{}(u^n)\big)\,dr\right)_{H^s}\right|\\
&:=B_5+B_6.
\end{align*}
Since $S(t_{n},r)-I$ is independent of $g(u^n)$,
\begin{align*}
B_5 
&\lesssim  (\delta t)^{\frac 12} \E(\|e_n\|_{H^s}^2)^{\frac12}\left(\int_{t_n}^{t_{n+1}}\E\left(\|\E\big(S(t_{n},r)-I\big)g_R^{}(u^n)\|_{H^s}^2 \right)dr\right)^{\frac 12}.
\end{align*}
Plancherel theorem combining the facts that 
\begin{align*}
\E(\exp(i(\beta(r)-\beta(t_n))|\xi|^2))=\exp \left( -\frac{(r-t_n)}2|\xi|^4\right),\quad 
|\exp(x)-1|\lesssim |x|
\end{align*} 
yield
\begin{align*}
&\E\left(\|\E(S(t_{n},r)-I)g_R^{}(u^n)\|_{H^s}^2 \right)\\
&\lesssim  \E\left( \int_{\R}|\E(\exp (\bi(\beta(r)-\beta(t_n))|\xi|^2)-1|^2|\hat g_R^{}(u^n)|^2(1+\xi^2)^s\,d\xi\right)\\
&\lesssim (r-t_n)^2\left(1+\E(\|u^n\|_{H^{s+4}}^2)\right), 
\end{align*}
from which we have
\begin{align*}
B_5
&\lesssim (\delta t)^{\frac12}\E(\|e_n\|_{H^s}^2)^{\frac12}
\left( \int_{t_n}^{t_{n+1}} (r-t_n)^2\left(1+\E(\|u^n\|_{H^{s+4}}^2)\right) dr\right)^{\frac12}\\
&\lesssim (\delta t)^2\E(\|e_n\|_{H^s}^2)^{\frac12}(1+\E(\|u_0\|_{H^{s+4}}^2)^{\frac12})\\
&\lesssim \delta t\,\E(\|e_n\|_{H^s}^2)
+(\delta t)^3 \left(1+\E(\|u_0\|_{H^{s+4}}^2)\right).
\end{align*}
Similarly,
\begin{align*}
B_6
\lesssim (\delta t)^2\E(\|e_n\|_{H^s}^2)^{\frac12}.
\end{align*}
Therefore,  
\begin{align}\label{t6}
|\E(e'_{n+1,1},\,e'_{n+1,3})_{H^s}|
\lesssim \delta t\,\E(\|e_n\|_{H^s}^2)
+(\delta t)^3 \left(1+\E(\|u_0\|_{H^{s+4}}^2)\right).
\end{align}
Combining  \eqref{t1}--\eqref{t6}, we deduce 
\eqref{k-od-spl}. 
The Gronwall inequality completes the proof.
\qed
\end{proof}

\begin{rk}
Our above arguments can also be applied to presenting a rigorous proof of first order of convergence for the pseudo-spectral splitting scheme 
\begin{align*}
\tilde u_R^n=\exp(\bi \,\delta W_n \,\Delta) u_R^n,\quad
u_R^{n+1}=\exp(\bi \,\delta t\, \theta_R(\tilde u_R^{n})|\tilde u_R^{n}|^2)\tilde u_R^{n},
\end{align*}
proposed in \cite{Mar06}.
\end{rk}

We now derive the order in probability of the schemes \eqref{mid-nls} and \eqref{t-spl} in terms of Proposition \ref{ts-order} and Proposition \ref {od-spl}.
We note that this method is also used in \cite{BDD15} to study a Crank-Nicolson scheme for Eq. \eqref{s-nls}.
To make it clear, denote by $\{u_{\bf{mid}}^n\}_{0\le n \le N}$ the solution of the scheme \eqref{mid-nls}  and $\{u_{\bf{Lie}}^n\}_{0\le n \le N}$ the solution of the scheme \eqref{t-spl}, respectively.

\begin{tm}\label{pod}
\hspace{10cm}
\begin{enumerate}
\item
Assume that $u_0\in H^{s+10}(\R)$. 
Then for any $0\leq n\leq N$,
\begin{align}\label{ord-mid}
\lim_{M\to \infty}\mathbb P \Big(\|u(t_n)-u_{\bf{mid}}^n\|_{H^s}\ge M \delta t\Big)=0.
\end{align}

\item
Assume that $u_0\in H^{s+4}(\R)$.
Then for any $0\leq n\leq N$,
\begin{align}\label{ord- Lie}
\lim_{M\to \infty}\mathbb P \Big(\|u(t_n)-u_{\bf{ Lie}}^n\|_{H^s}\ge M\delta t\Big)=0.
\end{align}
\end{enumerate}
\end{tm}

\begin{proof}
We only prove the estimate \eqref{ord- Lie} and the same arguments can be applied to \eqref{ord-mid}.
For simplicity, we omit the index and define a stopping time
\begin{align*}
\tau_R= \inf_{0\leq n\leq N} \left\{t_n:\ \left\| u_R^{n-1} \right\|_{H^s}\ge R
\quad  \text{or} \quad \left\| u_R^n \right\|_{H^s}\ge R
\right\},
\end{align*}
and the discrete solution $u^n_{\delta t}=u^n_{R}$ if $t_n \le\tau_R$.
By Proposition \ref{od-spl}, we have for any $R>0$,
\begin{align*}
\E\left(\sup_{0\le n\le N}\|u_R(t_n)-u^n_R \|_{H^s}^2 \right)
\le \sum_{0\le n\le N}\E\left(\|u_R(t_n)-u^n_R \|_{H^s}^2\right)
\le C(R)T(\delta t).
\end{align*}
This implies that $\sup_{0\le n\le N}\|u_R(t_n)-u^n_R \|_{H^s}$ converges to $0$ in probability as $\delta t\rightarrow 0$ and in turn yields $\lim\limits_{K \to \infty} \PP\left(\sup\limits_{n\le \N}\|u^n_{\delta t}\|\ge K\right)=0$(see \cite{BDD15}).
Applying Chebyshev inequality, we have 
\begin{align*}
&\PP\Big(\|u(t_n)-u^n_{\delta t} \|_{H^s}\ge M\delta t\Big)\\
&\le \PP\left(\sup_{0\le n\le N}\| u(t_n)\|_{H^s}\ge R\right)
+\PP\left(\sup_{0\le n\le N}\| u^n_{\delta t}\|_{H^s}\ge R\right)
+\PP\Big(\|u_R(t_n)-u^n_R \|_{H^s}\ge M\delta t\Big)\\
&\le \PP\left(\sup_{0\le n\le N}\| u(t_n)\|_{H^s}\ge R\right)+\PP\left(\sup_{0\le n\le N}\| u^n_{\delta t}\|_{H^s}\ge R\right)+\frac {\E(\|u_R(t_n)-u^n_R\|_{H^s})^2}{M^2(\delta t)^2},
\end{align*}
which converges to 0 as $R$ and $M$ tend to $\infty$ by
the boundedness of $u^n_{\delta t}$ and $u(t_n)$ together with Proposition \ref{od-spl}.
\qed
\end{proof}

\begin{rk}
If the high dimension problem is well-posed, the above arguments shows that the convergence order of the 
 mid-point scheme \eqref{mid-nls} and the  Lie splitting scheme \eqref{t-spl} is still one in probability. 
\end{rk}

\subsection{Multi-symplectic scheme}

To construct stochastic multi-symplectic integrators, we use the classical centered finite difference in spatial direction under homogenous Dirichlet boundary conditions combined with the mid-point scheme \eqref{mid-nls} in the temporal discretization.
We remark that although the temporal splitting scheme is explicit, its corresponding full discretization which uses the spatial centered finite difference could not preserve the multi-symplectic structure of the original equation.  Indeed,
after this Lie splitting approach, Eq. \eqref{spl2} possesses the multi-symplectic  conservation law. However, the fact that Eq. \eqref{spl1} only preserves the symplectic structure can not guarantee this splitting-based full discretization to inherit  the multi-symplectic  conservation law.
 
We work on the spatial domain $[-L_x,L_x]$ and use the uniform mesh generation.  
The spatial mesh step is $\delta x= \frac {2L_x}{N_x}$, and the time step is again $\delta t$. 
Then the temporal and spatial grid points are
$$(t_n, x_j)=(n\delta t, -L_x+j\delta x),\quad (n,j)\in \{0,1\cdots,N\}\times  \{0,1,\cdots,N_x\}.$$
For convenience, denote $z=(p,q,v,w), v=p_x, w=q_x$ and define
\begin{align*}
&\delta_t^+z^n:=\frac{z^{n+1}-z^n}{\delta t},\quad 
\delta_x^+z_j:=\frac{z_{j+1}-z_j}{\delta x},\quad 
\delta_x^-z_j:=\frac{z_j-z_{j-1}}{\delta x}.
\end{align*}
The full discretization of Eq. \eqref{s-nls} is 
\begin{align}\label{full}
\bi\delta_t^+u_j^n+\frac {\delta W_n}{\delta t} \delta_x^+\delta_x^-(u_j^{n+\frac 12})
+|u_j^{n+\frac 12}|^{2\sigma}u_j^{n+\frac 12}=0.
\end{align}

For the multi-symplectic Hamiltonian system, this scheme is equivalent to 
\begin{align}\label{hfull}
M\delta_t^+z_j^n+K^+\delta_x^-z_j^{n+\frac12}\frac {\delta W_n}{\delta t}+K^-\delta_x^+z_j^{n+\frac 12}\frac {\delta W_n}{\delta t}
=\nabla_zS_1(z_j^{n+\frac12})+\nabla_zS_2(z_j^{n+\frac12})\frac {\delta W_n}{\delta t},
\end{align}
where
\begin{align*}
K^+=\left(\begin{array}{cccc}0 & 0 & 1 & 0\\0 & 0 & 0 & 1\\0 & 0 & 0 & 0 \\0 & 0 & 0 & 0\end{array}\right),
\qquad  K^-=\left(\begin{array}{cccc}0 & 0 & 0 & 0\\0 & 0 & 0 & 0\\-1 & 0 & 0 & 0 \\0 & -1 & 0 & 0\end{array}\right).
\end{align*}
As expected, the full discretization \eqref{hfull} preserves the discrete multi-symplectic conservation law and the discrete charge conservation law.

\begin{tm}\label{fcl}
The scheme \eqref{full} is a stochastic multi-symplectic integrator, i.e., it satisfies the discrete stochastic multi-symplectic conservation law:
\begin{align}\label{mul-sym}
\delta_t^+(dz_j^n\wedge M^+dz_j^n)+\delta_x^+(dz_{j-1}^{n+\frac12}\wedge K^-dz_j^{n+\frac 12})\frac {\delta W_n}{\delta t}=0,\quad \text{a.s.}  
\end{align}
where $$M^+=\left(\begin{array}{cccc}0 & -1 & 0 & 0\\0 & 0 & 0 & 0\\0 & 0 & 0 & 0 \\0 & 0 & 0 & 0\end{array}\right).$$
 Meanwhile, the scheme \eqref{full} preserves the discrete charge conservation law:
\begin{align}\label{mul-ccl}
\delta x\sum_j |u_{j}^{n+1}|^2=\delta x\sum_j|u_j^{n}|^2, \quad \text{a.s.} 
\end{align}

\end{tm}

\begin{proof}
Taking differential in the phase space on \eqref{hfull}, we have
\begin{align*}
&M\delta_t^+dz_j^n+K^+\delta_x^-dz_j^{n+\frac12}\frac {\delta W_n}{\delta t}+K^-\delta_x^+dz_j^{n+\frac 12}\frac {\delta W_n}{\delta t}\\
&=\nabla_{zz}S_1(z_j^{n+\frac12})dz_j^{n+\frac12}+\nabla_{zz}S_2(z_j^{n+\frac12})dz_j^{n+\frac12}\frac {\delta W_n}{\delta t}.
\end{align*} 
Then wedging the above equality by  $dz_j^{n+\frac12}$, we obtain
\begin{align*}
dz_j^{n+\frac12}\wedge M\delta_t^+dz_j^n+dz_j^{n+\frac12} \wedge K^+\delta_x^-dz_j^{n+\frac12}\frac {\delta W_n}{\delta t}+dz_j^{n+\frac12}\wedge K^-\delta_x^+dz_j^{n+\frac 12}\frac {\delta W_n}{\delta t}=0.
\end{align*}
The fact that $dz_j^{n+\frac12}=\frac 12(dz_j^{n}+dz_j^{n+1})$ combined with the skew-symmetry of $M$ leads to the temporal symplectic structure
\begin{align*}
dz_j^{n+\frac12}\wedge M\delta_t^+dz_j^n
=\frac12(dz_{j+1}^n\wedge M d\delta_t^+dz_j^n+dz_j^n\wedge M d\delta_t^+dz_j^n)
=\delta_t^+(dz_j^n\wedge M^+dz_j^n).
\end{align*}
Next we deal with the spatial symplectic structure.
Due to the the skew-symmetry of $K$, we have
\begin{align*}
&dz_j^{n+\frac12} \wedge K^+\delta_x^-dz_j^{n+\frac12}\frac {\delta W_n}{\delta t}+dz_j^{n+\frac12}\wedge K^-\delta_x^+dz_j^{n+\frac 12}\frac {\delta W_n}{\delta t}\\
&=\frac {\delta W_n}{\delta t} \left( dz_j^{n+\frac12} \wedge K^+\delta_x^+dz_{j-1}^{n+\frac12}+dz_j^{n+\frac12}\wedge K^-\delta_x^+dz_j^{n+\frac 12} \right)\\
&=\frac {\delta W_n}{\delta t} \left( \delta_x^+dz_{j-1}^{n+\frac12} \wedge K^-dz_j^{n+\frac12}+dz_j^{n+\frac12}\wedge K^-\delta_x^+dz_j^{n+\frac 12} \right)\\
&=\delta_x^+\left( dz_{j-1}^{n+\frac12}\wedge K^-dz_j^{n+\frac 12} \right)\frac {\delta W_n}{\delta t}.
\end{align*}
We conclude \eqref{mul-sym} by summing up the above temporal and spatial symplectic structures.

Multiplying the full discretization \eqref{full} by $\bar u_j^{n+\frac12}$, summing over all spatial grids and then taking the imaginary part, we conclude \eqref{mul-ccl} combined with $\delta W_n$ is real valued.
\qed
\end{proof}
\begin{rk}
	The temporal discretization error for Eq. \eqref{full} is obtain in Theorem \ref{pod}, i.e, order one in probability. By the truncated argument, the spatial discretization error is the same as in the deterministic case (see e.g. \cite {BDD15}).
\end{rk}

\section{Numerical Experiments}
\label{exp}

One purpose of this section, via simulating the temporal orders of convergence of the mid-point scheme \eqref{mid-nls} and the splitting scheme \eqref{nspl} both with spatial centered difference discretization, is to verify the theoretical results in Theorem \ref{pod}. 
Another purpose is to show the good longtime behavior of the stochastic symplectic and multi-symplectic schemes.

We consider the schemes \eqref{mid-nls} and \eqref{nspl} with $\sigma=1$ for simplicity. 
The numerical spatial domain is $[-L_x,L_x]=[-30,30]$ and the initial datum is chosen to be Gaussian: $u(0,x)=\exp(-3x^2)$, 
$x\in [-30,30]$. 
We take the spatial mesh $\delta x = 0.05$ and compute a reference solution $u_{ref}$ on a fine mesh with $\delta t= 2^{-16}$.
 In Fig. \ref{fig-ord}, we plot the convergence curves based on the errors 
 $\|u_{ref}-u_{\delta t}\|_{L^2}$
 at time $T=0.5$ with $\delta t=2^p \delta t_{ref}$, $p=3,\cdot\cdot\cdot,7$. We can see that the slopes of the our schemes are both close to 1.
 This observation verifies the theoretical results in Section \ref{num}.

 \begin{figure}
\centering
\includegraphics[width=4in]{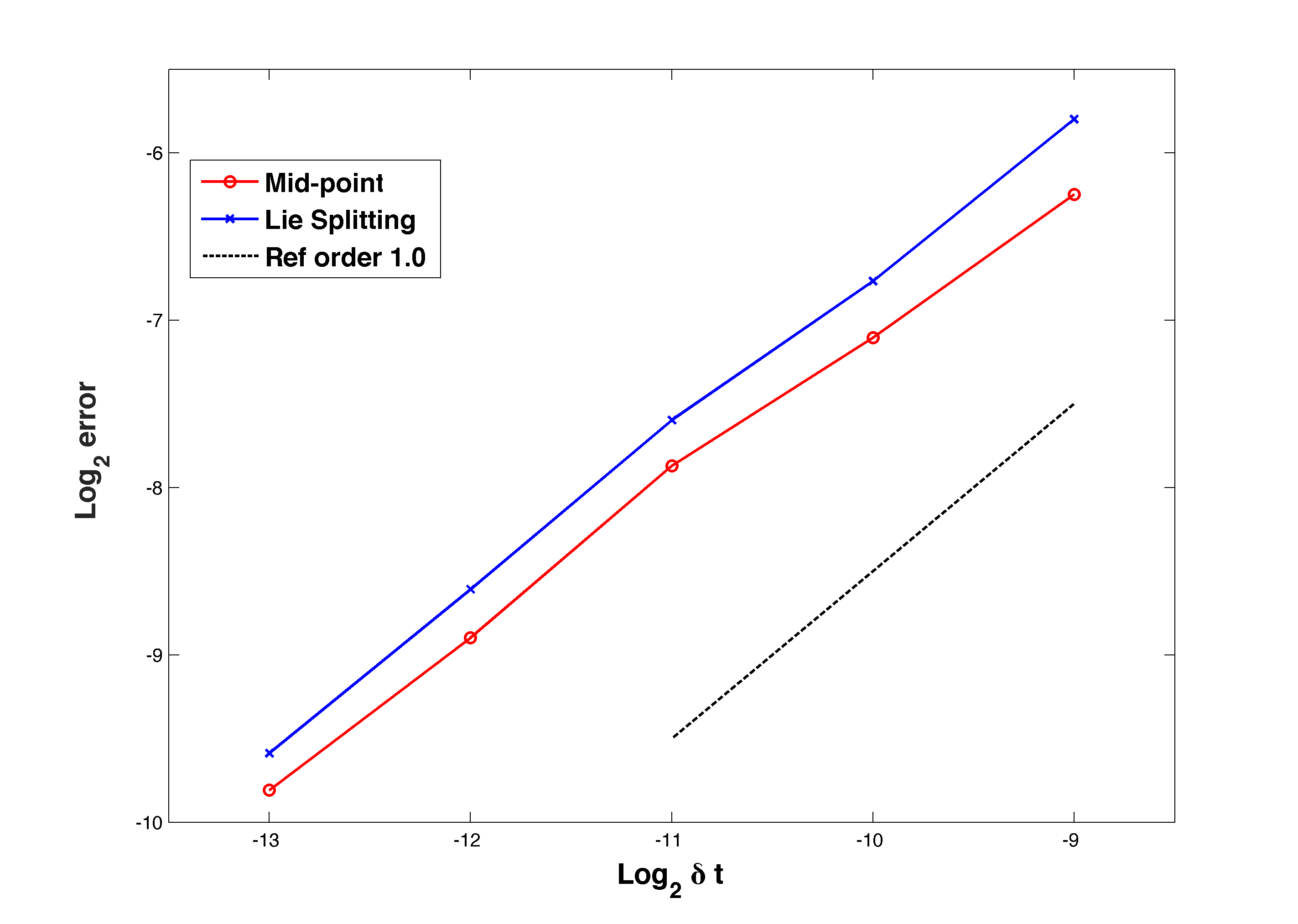}\\
\caption{Rates of convergence of the mid-point scheme (blue) and the Lie splitting scheme (red) for NLS with white dispersion}
\label{fig-ord}
\end{figure}

\iffalse
 \begin{figure}
\begin{minipage}[t]{0.5\linewidth}
\centering
\includegraphics[width=2.4in]{nonlinear1.eps}\\
\includegraphics[width=2.4in]{nonlinear2.eps}\\
\includegraphics[width=2.4in]{nonlinear3.eps}\\
\end{minipage}%
\begin{minipage}[t]{0.5\linewidth}
\centering
\includegraphics[width=2.4in]{othernonlinear1.eps}\\
\includegraphics[width=2.4in]{othernonlinear2.eps}\\
\includegraphics[width=2.4in]{othernonlinear3.eps}\\
\end{minipage}
\caption{Three sample trajectories of the solution of \eqref{non} obtained by pseudo-symplectic mid-point method, symplectic mid-point method, and backward Euler methods with $h=0.02$.}
\label{fig:3.3}
\end{figure}
\fi

 \begin{figure}
\begin{minipage}[t]{0.5\linewidth}
\centering
\includegraphics[width=2.5in]{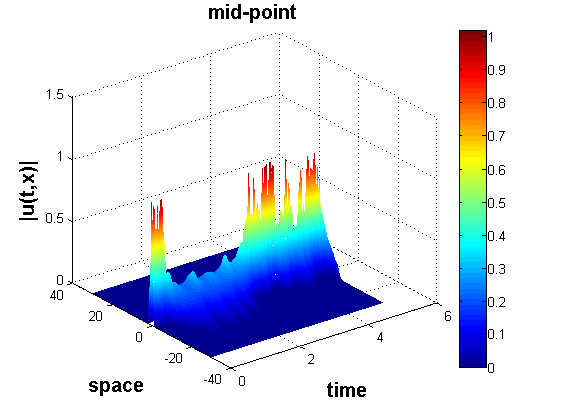}\\
\includegraphics[width=2.5in]{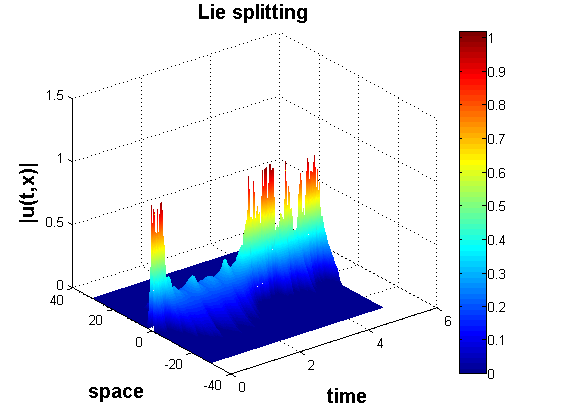}\\
\end{minipage}%
\begin{minipage}[t]{0.5\linewidth}
\centering
\includegraphics[width=2.5in]{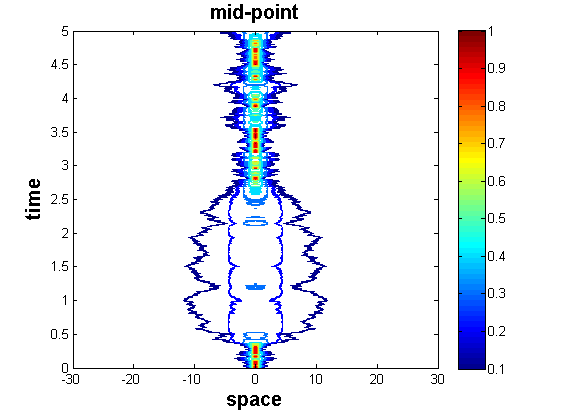}\\
\includegraphics[width=2.5in]{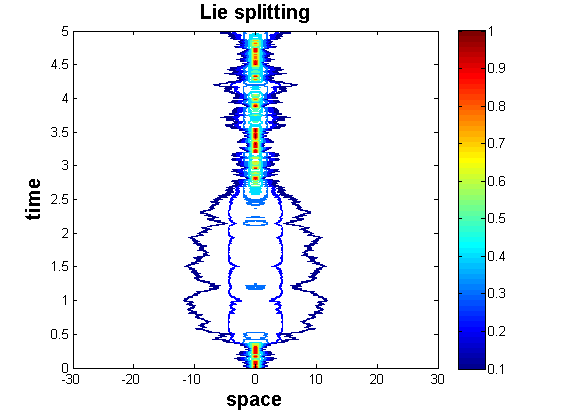}\\
\end{minipage}
\caption{The profile of the space-time evolution of $|u(t,x)|$ (left) and contour plot of $|u(t,x)|$ (right) for the stochastic multi-symplectic scheme \eqref{mid-nls} (up) and the stochastic symplectic splitting scheme \eqref{nspl} (down) in short time.}
\label{fig-short}
\end{figure}

 \begin{figure}
\begin{minipage}[t]{0.5\linewidth}
\centering
\includegraphics[width=2.5in]{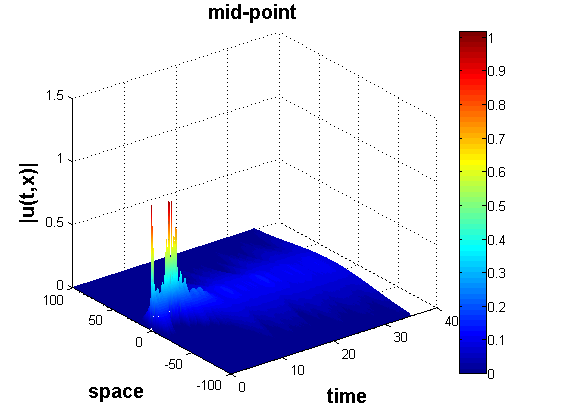}\\
\includegraphics[width=2.5in]{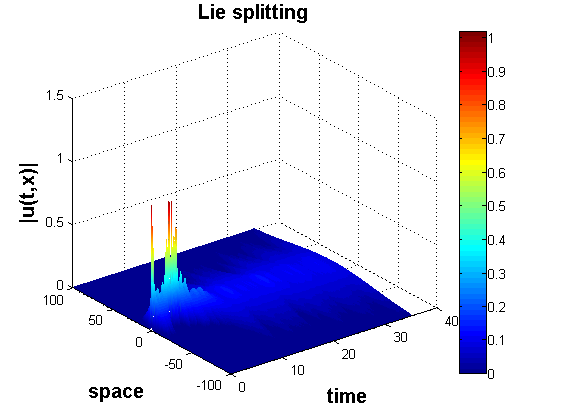}\\
\end{minipage}%
\begin{minipage}[t]{0.5\linewidth}
\centering
\includegraphics[width=2.5in]{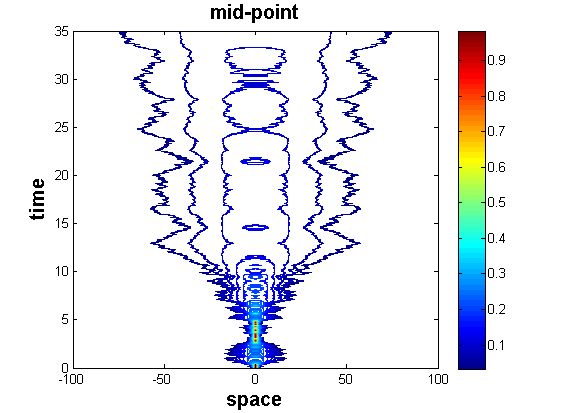}\\
\includegraphics[width=2.5in]{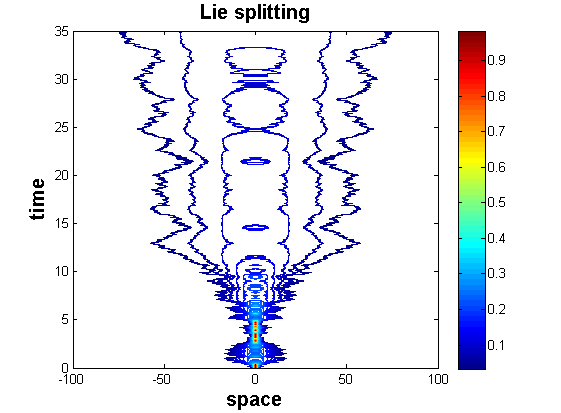}\\
\end{minipage}
\caption{The profile of the space-time evolution of $|u(t,x)|$ (left) and contour plot of $|u(t,x)|$ (right) for the stochastic multi-symplectic scheme \eqref{mid-nls} (up) and the stochastic symplectic splitting scheme \eqref{nspl} (down) in longtime.}
\label{fig-long}
\end{figure}

We now investigate the propagation of the solutions by our schemes. 
In this part we take $\delta x=0.05$ and $\delta t=2^{-12}$. The profiles of the amplitude $|u(t,x)|$ by the schemes are presented in Fig. \ref{fig-short}. 
This figure shows that dispersion and nonlinearity stay well balanced in a short time interval. 
Furthermore, Fig. \ref{fig-long} demonstrates the longtime behavior for these schemes.
We find that the amplitudes of  numerical 
solutions contract and expand alternatively,  and that the behavior of the numerical 
solutions is dominated by the dispersion after nearly $T=10$.     
The authors in \cite{BDD15} obtain similar numerical observations with a Crank-Nicolson scheme. 

 \begin{figure}
\subfigure{\includegraphics[width=.32\textwidth]{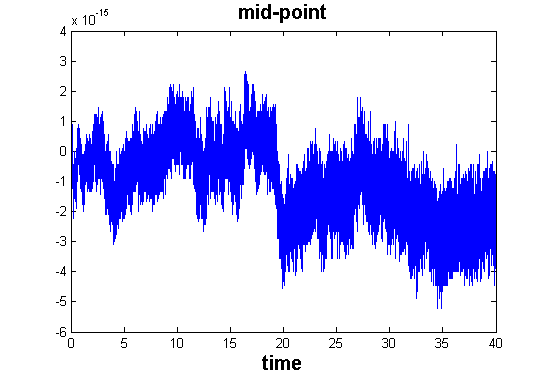}}
\subfigure{\includegraphics[width=.32\textwidth]{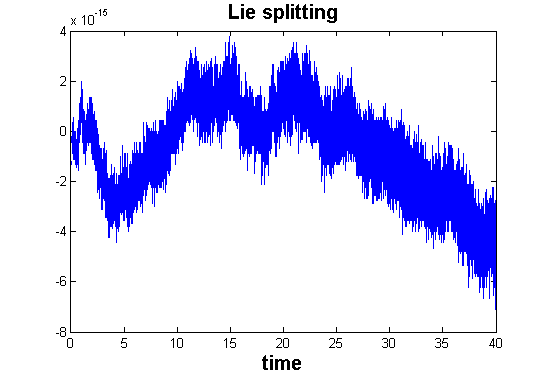}}
\subfigure{\includegraphics[width=.32\textwidth]{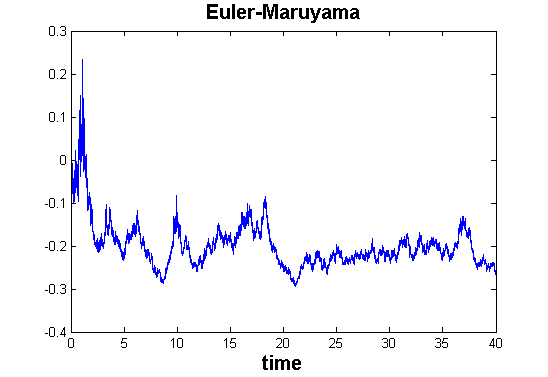}}
\caption
{The global errors of charge conservation law for mid-point scheme (left),  Lie splitting scheme (mid) and Euler-Maruyama scheme (right).}
\label{fig-L2}
\end{figure}

The preservation of the charge conservation law can be used to measure
the longtime behavior of the numerical scheme. For comparison with our stochastic symplectic and multi-symplectic schemes, we consider the non-symplectic Euler-Maruyama scheme:
\begin{align*}
&\bi\frac {u^{n+1}-u^n}{\delta t}+\frac {\delta W_n}{\delta t} \Delta u^{n}+g(u^{n})=0,
\end{align*}
with the same spatial discretization.
Fig. \ref{fig-L2} displays the evolution of global errors of the discrete charge conservation law
$\text{err}(n):= \delta x(\sum_{1\le j \le N_x} |u_{j}^{n}|^2-\sum_{1\le j \le N_x} |u_{j}^{0}|^2)$, $n\le N$ for the three schemes.
It turns out that the mid-point scheme \eqref{mid-nls} and the splitting scheme \eqref{nspl} both preserve the discrete charge 
conservation law exactly, which indicates that  the proposed two schemes have superiority in longtime computation
compared with the non-symplectic method.

\iffalse
We finally plot in Fig.\eqref{max} the evolution with time of expected amplitude, i.e.,
 $$\frac 1K \sum_{k=1}^K \max_j|u^k(t,x_j)|.$$
It is the numerical approximation of $\E (\max_x|u(t,x)|)$ where $K$
is the number of trajectories.
 \fi
\section{Conclusions}
\label{con}
In this paper, we focus on the stochastic NLS with white noise dispersion which is a representative stochastic Hamiltonian PDE. 
Based on the stochastic symplectic and stochastic multi-symplectic structures, we propose a symplectic splitting scheme and a multi-symplectic scheme whose temporal order are both one in probability. Moreover, the two schemes have good qualitative properties in longtime computations, and both preserve the discrete charge conservation law. 

\iffalse
we present the stochastic multi-symplecticity of the stochastic Hamiltonian PDEs, which extends the scope of the multi-symplecticity of Hamiltonian PDEs.
\fi 

\section*{Acknowledgments}
The authors gratefully thank the anonymous referees for valuable comments and suggestions in improving this paper. This work was supported by National Natural Science Foundation of China (No. 91630312, No. 91530118 and No. 11290142).

\bibliographystyle{plain}
\bibliography{bib}

\end{document}